\documentclass[10pt,a4paper]{article}

\usepackage{amsmath}
\usepackage{amsfonts}
\usepackage{amssymb}
\usepackage{graphicx}
\usepackage{amsthm}
\usepackage[all]{xy}
\usepackage{lineno,hyperref}
\usepackage{comment}
\usepackage[subnum]{cases}
\usepackage{multicol}
\usepackage{multirow}

\newcounter{oftheorem}[subsection]
\newenvironment{mytheorem}[1]%
{\begin{trivlist}
     \renewcommand{\theoftheorem}{#1~\thesection.\arabic{oftheorem}}
     \refstepcounter{oftheorem}
     \item[\hspace{\labelsep}\bf\thesection.\arabic{oftheorem} #1]}%
{\end{trivlist}}
\newenvironment{definition}{\begin{mytheorem}{Definition}\it}{\end{mytheorem}}
\newenvironment{example}{\begin{mytheorem}{Example}\it}{\end{mytheorem}}

\newenvironment{theorem}{\begin{mytheorem}{Theorem}\it}{\end{mytheorem}}

\newenvironment{remark}{\begin{mytheorem}{Remark}}{\end{mytheorem}}
\newenvironment{lemma}{\begin{mytheorem}{Lemma}}{\end{mytheorem}}

\author{Stavros Anastassiou\\
Department of Mathematics,\\ University of West Macedonia\\
GR-52100 Kastoria, Greece\\sanastassiou@gmail.com}

\title{Local models for smooth vector fields of the line}
\begin{document}
\maketitle
\begin{abstract}
We present the local classification of singularities of smooth vector fields on the line, with respect to the equivalence relation of $C^1$--conjugacy. Along the way, we recall the analogous classification, up to $C^0$ and $C^{\infty}$ conjugacy. We also give the transversal unfoldings of the corresponding normal forms and treat the case where the changes of coordinates are tangent to the identity. Thus, a fairly complete description of the $1$--d case is achieved.   
\end{abstract}
\section{Introduction}
The problem of classifying vector fields, according to the qualitative structure of their orbits, has a long history. In this problem, one has to prove the existence of a $C^r,\ r\geq 0$, invertible mapping, transforming orbits of one vector field to orbits of another, respecting the time--parametrization of those orbits (or just their orientation).

In a neighborhood of a hyperbolic singularity, the classification problem is practically solved. Up to $C^0$ invertible mappings, the orbit structure of a vector field is, locally, identical with the phase structure of its linear part, according to the Grobman--Hartman theorem. Up to $C^{\infty}$--diffeomorphisms, Sternberg's theorem assures, once again, the linearization of the vector field, as long as the eigenvalues of its linear part satisfy a number of ``non--resonant conditions". Classical references on singularities of vector fields include \cite{Sternberg,Takens,Dumortier}, where more details can be found.       

To find ``simple" models of vector fields at non--hyperbolic singularities, up to $C^{\infty}$ invertible maps, one usually resorts to normal form theory. The terms that can be ``removed", using smooth transformations, are identified and the remaining terms consist what is known as the ``normal form" of the vector field. Then, one has to study perturbations of this field, to give a complete picture of the bifurcations it undergoes. The reader may consult e.g. \cite{Murdock} and \cite{Sanders-Verhulst-Murdock}, for an account of normal forms and unfoldings of vector fields, while \cite{Ren-Yang} contains results regarding the $C^1$ case. 

Perhaps not surprisingly, the $C^1$ case turns out to be somewhat more complicated, since the homological equations involving $C^1$ vector fields cannot be solved using methods available in the $C^{\infty}$ case. Actually, as shown in \cite{Helene1}, even in the $1$--dimensional case, $C^1$ vector fields having a hyperbolic singularity exist, which cannot be, locally, linearized.

In this article, we deal with the local study at the origin of $C^{\infty}$ vector fields of the line, the vector space of which will be denoted by $\mathcal{X}(\mathbb{R})$. We identify the vector field $f(x)\frac{\partial}{\partial x}$ with its coefficient $f$. So, let $\mathcal{E}$ stand for the ring of infinitely differentiable univariate function--germs. If $\phi$ is the germ at the origin of a $C^k$ reversible transformation, we wish to provide local models for vector fields under the usual conjugacy relation: vector fields $f,g$ are said to be $C^k$--conjugate if $g(x)=\frac{1}{\phi '(x)}f(\phi (x))$.   

In section 2, we give local models for singularities of vector fields, under this conjugacy relation, while in section 3 we consider the problem of their unfoldings. We completely solve the $C^1$ case, while recalling, along the way, the analogous statements for $C^0$ and $C^\infty$ invertible mappings. When the transformation $\phi$ is differentiable, we consider separately the case where it is tangent to the identity, that is, $\phi'(0)=1$.

Thus, in this way, a fairly complete picture of the local behaviour of $1$--d vector fields is presented.
\section{Local classification of members of $\mathcal{X}(\mathbb{R})$}
\subsection{The $C^0$ case}
Let $f(x)\frac{\partial}{\partial x},g(x)\frac{\partial}{\partial x}\in \mathcal{X}(\mathbb{R})$, with corresponding flows $f^t,g^t$. In the $C^0$ case, the change of coordinates is a local  hommeomorphism $\phi$ of $\mathbb{R}$. Since it is not, in general, differentiable, it acts not on the vector fields themselves, but on their flows. That is, the hommeomorphism is expected to map flow $f^t$ to $g^t$, i.e. $\forall t \in \mathbb{R}$, relation $\phi\circ f^t=g^t\circ \phi$ should hold.

In case the origin is not a fixed point for the flows $f^t,g^t$, the identity transformation can be used to show that these two flows are conjugate. 

We therefore turn our attention to the case where the origin is a fixed point for both flows. Denote by $f^t_+,g^t_+$ (similarly, $f^t_-,g^t_-$) the restrictions of the two flows on the positive (negative) semi--axis. Let us suppose that:
\begin{enumerate}
\item[(1)] {Either $\lim \limits_{t\rightarrow +\infty}f_+^t=\lim \limits_{t\rightarrow +\infty}g_+^t=0$, or $\lim \limits_{t\rightarrow -\infty}f_+^t=\lim \limits_{t\rightarrow -\infty}g_+^t=0$.}
\item[(2)] {Either $\lim \limits_{t\rightarrow +\infty}f_-^t=\lim \limits_{t\rightarrow +\infty}g_-^t=0$, or $\lim \limits_{t\rightarrow -\infty}f_-^t=\lim \limits_{t\rightarrow -\infty}g_-^t=0$.}
\end{enumerate}  
We have the following:
\begin{theorem}
Let us suppose that $f^t,g^t$ are flows on the line, having the origin as a fixed point. Assume that conditions (1) and (2) above hold. Then, the germ of a hommeomorphism $\phi$ exists, such that $\phi\circ f^t=g^t\circ \phi$.
\end{theorem}
\begin{proof}
Set $h(0)=0$.

Let us restrict our attention on the positive semi-axis and suppose that $\lim \limits_{t\rightarrow +\infty}f_+^t=\lim \limits_{t\rightarrow +\infty}g_+^t=0$ (the other case can be treated in exactly the same way).

We can suppose that both $f^t_+,g^t_+$ are defined on the interval $(0,\epsilon),\ \epsilon >0$, which interval contains no zeros for the fields $f,\ g$. Define functions $\tau_f,\tau_g:(0,\epsilon)\rightarrow \mathbb{R}$ as follows:
\[
\tau_f(x)=\int_{\epsilon/2}^x\frac{1}{f(y)}dy,\ \tau_g(x)=\int_{\epsilon/2}^x\frac{1}{g(y)}dy.
\]
Observe that they are well--defined, differentiable and invertible, since their derivatives do not vanish.

We also observe that:
\[
\tau_f\circ f_+^t(x)=\int_{\epsilon/2}^{f_+^t(x)}\frac{1}{f(y)}dy=\int_{\epsilon/2}^{x}\frac{1}{f(y)}dy+\int_{x}^{f_+^t(x)}\frac{1}{f(y)}dy=\tau_f(x)+t.
\]
To verify the last equality, just differentiate relation $t=\int \limits _{x}^{f_+^t(x)}\frac{1}{f(y)}dy$ with respect to $t$.

Similarly, $\tau_g\circ g_+^t=\tau_g+t$. If $T^t(x)=x+t$ is the usual standard translation, we have:
\[
T^t=\tau_f\circ f^t_+\circ \tau_f^{-1}=\tau_g\circ g^t_+\circ \tau_g^{-1}\Rightarrow \tau_g^{-1}\circ \tau_f \circ f^t \circ \tau_f^{-1} \circ \tau_g=g_+^t.
\]
The sought mapping is, of course, $\phi=\tau_g^{-1}\circ \tau_f$. It is obviously continuous on $(0,\epsilon)$, while $\lim \limits _{x\rightarrow 0}\phi(x)=0$, due to condition (1) above. 

We then extend $\phi$ to the left half--line in an analogous manner, under the assumption that condition (2) is satisfied and the theorem is proved.
\end{proof}

\begin{remark}
Note that, in the proof above, flows $f^t,g^t$ need to only be continuous. Thus, the result holds not only in case the vector fields are smooth, but even in the $C^1$ case. In all cases, the differentiability of the conjugating hommeomorphism $\phi$ is not guaranteed. 
\end{remark}

\begin{example}
The flows of the, smooth, vector fields $x\frac{\partial}{\partial x},2x\frac{\partial}{\partial x}$ are conjugated, even globally, via $\phi:\mathbb{R}\rightarrow \mathbb{R}$, defined as:
\[
\phi(x)=
\begin{cases}
-x^2,\ &\ x<0\\
x^2,\ &\ x\geq 0
\end{cases}.
\]
Although $\phi$ is a hommeomorphism, its inverse is not differentiable at the origin.
\end{example}

Up to hommeomorphisms, there are thus four models for local vector fields, given in the following table.
\vspace*{0.3cm}
\begin{center}
Table 1:\\
Local models under topological conjugacy
\end{center}    
{\small
\begin{center}
\begin{tabular}{|c|c|}
\hline
Singularity & Vector field\\
\hline 
\hline
regular point & $\frac{\partial}{\partial x}$ \\ 
attracting singularity & $-x\frac{\partial}{\partial x}$ \\
repelling singularity & $x\frac{\partial}{\partial x}$ \\
degenerate singularity & $x^2\frac{\partial}{\partial x}$ \\
\hline
\end{tabular}
\end{center}}
\vspace*{0.2cm}
\begin{remark}
Vector field $-x^2\frac{\partial}{\partial x}$ is $C^0$ conjugate to $x^2\frac{\partial}{\partial x}$, via hommemorphism $\phi(x)=-x$.
\end{remark}
We have thus completed the topological classification of $1$--d vector fields.  
\subsection{The $C^1$ case}
We now turn our attention to the  local classification of members of $\mathcal{X}(\mathbb{R})$, up to the equivalence relation of  $C^1$--conjugacy, i.e., the conjugacy $\phi$ in relation $g(x)=\frac{1}{\phi '(x)}f(\phi (x))$ is supposed to be a diffeomorphism of class $C^1$. We prove in detail the case where the changes of variables are tangent to the identity. The corresponding results in the general case are proved in a similar manner. To achieve our classification, we use methods from singularity theory, see for example \cite{Brocker}.

Since we identify the vector fields with their coefficients, we work in $\mathcal{E}$ and denote by $m$ the ideal of functions vanishing at the origin and with $m^2$ the ideal of functions vanishing, along with their first derivatives, at the origin. 

Consider a curve of function germs $g_s,\ s\in \mathbb{R}$, passing through $f$ for $s=0$. Supposing that all points of $g_s$ are $C^1$--conjugate to $f$, there exists a curve of local $C^1$--diffeomorphisms $\psi_s:(\mathbb{R},0)\rightarrow (\mathbb{R},0)$, satisfying $\psi_0(x)=x,\ \psi_s(0)=0,\ \psi'_s(0)=1$, such that:
\[g_s(x)=\frac{1}{\psi_s '(x)}f(\psi_s(x)).\]

Taking derivatives with respect to $s$, in that last equality, and evaluating at $s=0$, we get:
$$\frac{\partial}{\partial s}g_s(x)|_{s=0}=-X'(x)f(x)+X(x)f'(x),$$
where the vector field $X$ is defined through relation $\frac{\partial}{\partial s}\psi_s(x)=X(\psi_s(x))$. Notice that, since $\psi'(0)=1$, equations $X(0)=X'(0)=0$ hold.
\begin{lemma}
Let $f\in \mathcal{E}$. The ideal generated by $-X'f+Xf',\ X(0)=X'(0)=0$, equals $f\cdot m+f'\cdot m^2$.
\end{lemma}
\begin{proof}
The inclusion $\langle -X'f+Xf'\rangle \subseteq f\cdot m+f'\cdot m^2$ is obvious, so we shall prove the opposite inclusion.

Let $h\in f\cdot m+f'\cdot m^2$. There exist $g\in m,\ k\in m^2$ such that $h=fg+f'k$. We wish to find a $X\in m^2$ such that $fg+f'k=-X'f+f'X$.

A solution of last equation is:
\[
X(x)=
\begin{cases}
k(x)-f(x)\int_0^x\frac{g(t)+k'(t)}{f(t)}dt,&\ x\neq 0\\
0,&\ x=0
\end{cases}.
\]
Observe that $X(0)=0$ and:
$$X'(0)=\lim_{v\rightarrow 0}\frac{X(v)-X(0)}{v}=\lim_{v\rightarrow 0}\frac{k(v)-f(v)\int _0^v\frac{g(t)+k'(t)}{f(t)}dt}{v}.$$
Since $k\in m^2$, it is $\lim_{v\rightarrow 0}\frac{k(v)}{v}=0$, while:
$$|\lim_{v\rightarrow 0}\frac{f(v)\int_0^v\frac{g(t)+k'(t)}{f(t)}dt}{v}|=\lim_{v\rightarrow 0}\frac{1}{|v|}\cdot|f(v)|\cdot|v|\cdot\frac{1}{|f(v)|}\cdot|g(v)+k'(v)|=0.$$
Thus $X'(0)=0$, ensuring that $X\in m^2$.
\end{proof}
In light of the above, we introduce the following:
\begin{definition}
Let $f$ be the germ of a function at the origin. The tangent space of $f$, with respect to $C^1$ and tangent to identity conjugacy, is defined to be $Tf=f\cdot m+f'\cdot m^2.$
\end{definition}
Now, a germ $f$ is called $k$--determined, with respect to $C^1$--conjugation, if every other germ, with the same $k$--jet, is conjugate to $f$ through a $C^1$--diffeomorphism which is tangent to the identity. If such finite $k$ does not exist, we say that $f$ is not finitely determined.
\begin{theorem}
Function--germ $f$ is $k$--determined if, and only if, $m^{k+1}\subseteq Tf$.
\end{theorem}
\begin{proof}
Let $m^{k+1}\subseteq Tf$. If $s\in [0,1]$ and $h\in m^{k+1}$, we consider the curve $f_s(x)=f(x)+sh(x)$. We seek a family of tangent to the identity diffeomorphisms, $\psi_s:(\mathbb{R},0)\rightarrow (\mathbb{R},0)$, such that, $\forall s\in [0,1],\ \frac{1}{\psi '(x)}f_s(\psi(x))=f(x).$

Taking derivatives with respect to $s$, we get:
$$-\frac{1}{(\psi_s'(x))^2}X'(\psi_s(x))\cdot f_s(\psi_s(x))+\frac{1}{\psi_s'(x)}\cdot f_s'(\psi_s(x))\cdot X(\psi_s(x))=$$
$$=-\frac{1}{\psi_s'(x)}h(\psi_s(x)),$$
where, as above, $X\in m^2$ is defined through relation $\frac{\partial}{\partial s}\psi_s(x)=X(\psi_s(x))$.

Since $\psi_s'(0)=1$, equation above becomes:
$$-\frac{1}{\psi_s'(x)}X'(y)f_s(y)+f_s'(y)X(y)=-h(y),$$
where we have set, for simplicity, $y=\psi_s(x)$. This equation is equivalent to equation:
$$-\tilde{X}'(y)f_s(y)+f_s'(y)\tilde{X}(y)=-h(y),$$
where:
$$\tilde{X}(y)=X(y)+f(y)\int_0^y\frac{X'(t)}{f(x)}\bigg(\frac{1}{\psi'(t)}-1\bigg)dt.$$
We have thus to show that, $\forall h\in m^{k+1}$ and $\forall s\in[0,1]$, equation
$$-\tilde{X}'(y)f_s(y)+f_s'(y)\tilde{X}(y)=-h(y)$$
can be solved with respect to $\tilde{X}\in m^2$.

Notice that, for $s=0$, the equation above does have a solution with respect to $\tilde{X}\in m^2$, since $h\in m^{k+1}\subseteq Tf$.

Let $\mathcal{R}$ be the ring of germs, along $\{0\}\times [0,1]$, of functions depending on two variables, $g(x,s)$, and denote by $m_{1}$ the ideal $\{g\in \mathcal{R}/g(0,s)=0\}=\langle x\rangle$. By assumption, we have that $h\in m_1^{k+1}$. We need to show that, $\forall s\in [0,1],\ m_1^{k+1}\subseteq Tf_s$.

It is:
$$m_1^{k+1}\subset Tf\cdot \mathcal{R}=Tf_s+Th\subseteq Tf_s+m_1^{k+2}=Tf_s+m_1\dot m_1^{k+1},$$
thus, by Nakayama's lemma, $m_1^{k+1}\subseteq Tf_s$, and we get our desired solution.

To prove the inverse, let $f$ be $k$--determined, with respect to $RK_1$--equivalence, and $h\in m^{k+1}$. Consider the curve $f_s(x)=f(x)+sh(x)$. There exists a family of tangent to the identity diffeomorphisms $\psi_s(x)$, such that:
$$f=\frac{1}{\psi_s'(x)}f_s(\psi_s(x)).$$
Differentiating, with respect to $s$, we get:
$$0=-\frac{1}{(\psi_s'(x))^2}X'(\psi_s(x))f_s(\psi_s(x))+\frac{1}{\psi_s'(x)}f'(\psi_s(x))X(\psi_s(x))=$$
$$=-\frac{1}{\psi_s'(x)}h(\psi_s(x)),$$
where, as always, $\frac{\partial}{\partial s}\psi_s(x)=X(\psi_s(x))$.

Fos $s=0$, the last equation becomes $-h(x)=-X'(x)f(x)+f'(x)X(x)$, thus $m^{k+1}\subseteq Tf$, since $X\in m^2$.   
\end{proof}
We are now ready to classify members of $\mathcal{E}$. We denote by ``$\sim$" the equivalence relation of $C^1$--conjugacy and by ``$\sim_1$" the equivalence relation of $C^1$--conjugacy, through diffeomorphisms tangent to the identity.
\begin{theorem}
Let $f\in \mathcal{E}$.
\begin{enumerate}
\item[(i)] {If $f(0)=a\neq 0$, then $f\sim 1$ and $f\sim_1a$.}
\item[(ii)] {If $f(0)=0,\ f'(0)=a\neq 0$, then $f\sim ax$ and $f\sim_1 ax$.}
\item[(iii)] {If $f(0)=..=f^{(k-1)}(0)=0,\ f^{(k)}(0)=a\neq 0,\ k\geq 2$, then $f\sim x^k$ and $f\sim_1 ax^k$.}
\end{enumerate}
\end{theorem}
\begin{proof}
\begin{enumerate}\item[]
\item[(i)] {Let $f(0)=a\neq 0$. The solution of equation $f=\frac{1}{\psi'}$ is, of course, $\psi(x)=\int_0^x\frac{1}{f(x)}dx$ which is of class $C^1$, in a neighborhood of zero and furthermore $\psi'(0)\neq 0$, thus $\psi$ is a local $C^1$ diffeomorphism. We have just proved that $f\sim 1$.

To show that $f\sim_1a$, we have to find a local diffeomorphism $\psi(x)=x+\psi_1(x)$, with $\psi_1(0)=\psi_1'(0)=0$, such that:
$$f(x)=\frac{a}{1+\psi_1'(y)}.$$
The solution of last equation is:
$$\psi_1(x)=\int_0^x\bigg(\frac{a}{f(t)}-1\bigg)dt,$$
ensuring that the desired diffeomorphism is $\psi(x)=\int_0^x\frac{a}{f(t)}dt$, which indeed fulfils conditions $\psi(0)=0,\ \psi'(0)=1$.}
\item[(ii)] {In this case, $Tf=m^2$, thus $f$ is $1$--determined. 

In case $ax\sim bx$, a local diffeomorphism exists, such that $ax=\frac{b\psi(x)}{\psi'(x)}$. For the solution $\psi(x)=cx^{\frac{a}{b}}$ to be a diffeomorphism, relation $a=b$ should hold.

If $ax\sim_1 bx$, a diffeomorphism $\psi(x)=x+\psi_1(y)$, with $\psi_1(0)=\psi_1'(0)=0$, exists such that:
$$ax=\frac{bx+b\psi_1(x)}{1+\psi_1'(x)}.$$
The solution of last equation, for $a\neq b$, is $\psi_1(x)=-y$, which violates the condition $\psi_1'(0)=0$. Therefore, $ax\sim_1 bx$ if, and only if, $a=b$.}
\item[(iii)] {In this case, $Tf=m^{k+1}$, ensuring that $f$ is $k$--determined. 

For $ax^k\sim bx^k$ to hold,  a diffeomorphism $\psi(x)$ exists, such that:
$$\frac{1}{\psi'(x)}a\psi^k(x)=bx^k.$$
The unique solution, satisfying $\psi(0)=0$, is $\psi(x)=(\frac{a}{b})^{1/1-k}x$.

To also have $\psi'(0)=1$ relation $a=b$ should hold.}
\end{enumerate}
\end{proof}
We can now provide, in Table 2, models for one--dimensional vector fields, under the $C^1$--conjugacy relation, when both the conjugating diffeomorphism is tangent to the identity, or not. The constant $a$, appearing in the table, is understood to be non--zero.  
\vspace*{0.3cm}
\begin{center}
Table 2:\\
Local models under $C^1$ conjugacy
\end{center}    
{\small \begin{center}
   \begin{tabular}{|c|c|c|}
   \hline
Singularity & General Case & Tangent to identity case \\
\hline 
\hline
regular point & $1$ & $a$ \\ 
hyperbolic singularity & $ax$ & $ax$\\
degenerate singularity & $x^k$ & $ax^k$ \\
\hline 
       \end{tabular}
        \end{center}}
\vspace*{0.2cm}
\subsection{The $C^k,\ k\geq 2$ case}
In the $C^{\infty}$ case, the classification of $1$--dimensional vector fields is well--known. We state the main result here and refer to \cite{Belitskii} for the proof.
\begin{theorem}\label{Theo-capeiro-normal-forms}
Let $f(x)\frac{\partial}{\partial x}$ the germ at the origin of a, non--flat, smooth vector field of the line. This field is, locally, smoothly conjugate either with the vector field $a x\frac{\partial}{\partial x}$, in case $f'(0)=a \neq 0$, or with the vector field $\big(\pm x^k+dx^{2k-1}\big)\frac{\partial}{\partial x},\ d\in \mathbb{R},k\geq 2$.  
\end{theorem}  
The sign of the term $x^k$, in the coefficient $\pm x^k+dx^{2k-1}$ depends on $k$: if $k$ is odd, the sign is equal to "+", otherwise the sign is equal to the sign of $a$. If one chooses the conjugating diffeomorphism to be tangent to the identity, the coefficient of the term $x^k$ cannot be removed: in this case, the normal form is $ay^k+dy^{2k-1},a\neq 0$.

In the $C^k,\ k\geq 2$ but finite, the vector field can be brought to the same normal form, with the possible addition of flat terms. We refer the reader to \cite{Belitskii2} for details.

\begin{example}
Consider the 1--d vector field $(x^2+x^3)\frac{\partial}{\partial x}$ and the local function
\[
\phi(x)=
\begin{cases}
\frac{x}{1+x\log x-x\log(x+1)},\ & x\neq 0\\
0,\ & x=0
\end{cases}.
\]
As easily confirmed, $\phi$ is a, tangent to the identity, local diffeomorphism of class $C^1$, not even $C^2$. Relation $x^2+x^3=\frac{\phi^2(x)}{\phi'(x)}$ holds, ensuring that our vector field is $C^1$, but not $C^2$, conjugate to the vector field $x^2\frac{\partial}{\partial x}$, in accordance with the results presented above.    
\end{example}
We tabulate below the local models in the case of $C^{\infty}$ conjugacy.
\vspace*{0.3cm}
\begin{center}
Table 3:\\
Local models under $C^{\infty}$ conjugacy
\end{center}    
{\small \begin{center}
   \begin{tabular}{|c|c|c|}
   \hline
Singularity & General Case & Tangent to identity case \\
\hline 
\hline
regular point & $1$ & $a$ \\ 
hyperbolic singularity & $ax$ & $ax$\\
degenerate singularity & $\pm x^k+dx^{2k-1}$ & $a x^k+dx^{2k-1}$ \\
\hline 
       \end{tabular}
        \end{center}}
\vspace*{0.2cm}
Having completed the classification of germs of $1$--d vector fields, we now focus on their bifurcations.
\section{Unfoldings for $1$--d vector fields}
The concept of an unfolding relies on the computation of tangent spaces of equivalence classes; thus we consider only the $C^1$ and the $C^{\infty}$ case. 

We also restrict ourselves to the case of degenerate singular points since, in case the origin is a regular point or a non--degenerate singular point for our vector field, it is not affected by perturbations. 
\subsection{The $C^1$--case}
In the $C^1$ case, the main result concerning transversal unfoldings is the following:
\begin{theorem}
\begin{enumerate}\item[]
\item {In the general case, a transversal unfolding of the vector field $x^k\frac{\partial}{\partial x},\ k\geq 2$ is given by the family $Q:\mathbb{R}\times \mathbb{R}^{k-1}\rightarrow \mathbb{R}$, defined as follows:
\[
Q(x,\lambda_1,..,\lambda_{k-1}):=\big(x^k+\sum _{i=1}^{k-1}\lambda_{i}x^i\big)\frac{\partial}{\partial x}.
\]
}
\item {In case the changes of coordinates are tangent to the identity, a transversal unfolding of the vector field $ax^k\frac{\partial}{\partial x},\ k\geq 2$ is given by the family $Q_1:\mathbb{R}\times \mathbb{R}^k\rightarrow \mathbb{R}$, defined as follows::
\[
Q_1(x,\lambda_1,..,\lambda_{k}):=\big((a+\lambda_k)x^k+\sum _{i=1}^{k-1}\lambda_{i}x^i\big)\frac{\partial}{\partial x}.
\]
}
\end{enumerate}
\end{theorem}
\begin{proof}
In the tangent to the identity case, it is $T(x^k)=m^{k+1}$, thus $m/T(x^k)=\langle x,..,x^k\rangle$. All derivatives $\frac{\partial Q_1}{\partial \lambda _i}(\lambda_1,...,\lambda_k)$ belong to $m/T(x^k)$, confirming our statement.

In the general case, multiples of $x^k$ are conjugate to $x^k$. Parameter $\lambda_k$ is, therefore, not needed.
\end{proof}
\subsection{The $C^{\infty}$ case}
In the $C^{\infty}$ case, the unfoldings of vector fields of the line are given in the following:
\begin{theorem}\label{theo-Kostov}
\begin{enumerate}\item[]
\item {In the general case, a versal unfolding of the vector field $\big(\pm x^k+d x^{2k-1}\big)\frac{\partial}{\partial x}$ is given by $F:\mathbb{R}\times \mathbb{R}^k\rightarrow \mathbb{R}$, defined as:
\[
F(x,\lambda_1,..,\lambda_k)=\big(\pm x^k+\sum_{i=1}^{k-1}\lambda_ix^{k-1-i}+dx^{2k-1}\big)\frac{\partial}{\partial x}.
\]}
\item {In the tangent to the identity case, a versal unfolding of the vector field $\big(ax^k+d x^{2k-1}\big)\frac{\partial}{\partial x}$ is given by $F_1:\mathbb{R}\times \mathbb{R}^k\rightarrow \mathbb{R}$, defined as:
\[
F_1(x,\lambda_1,..,\lambda_k)=\big(ax^k+\sum_{i=1}^{k-1}\lambda_ix^{k-1-i}+dx^{2k-1}\big)\frac{\partial}{\partial x}.
\]}
\end{enumerate}
\end{theorem}
  For a proof of the theorem above, we refer the reader to \cite{Kostov,Klimes-Rousseau}. 
 \section{Conclusions}
We have studied local models and unfoldings for smooth vector fields of the line. We reviewed the results for the cases of topological and smooth conjugacy and we completely solved the problem for the $C^1$ case. Special attention was given to the case where the changes of coordinates are tangent to the identity. We hope that the reader will find useful this, fairly complete, description of the classification problem in dimension $1$.

The classification of vector fields is, of course, a subject that continues to attract the interest of many researchers (see \cite{Belitskii,Murdock}), while the differentiability properties of the conjugating transformation is of special importance in problems of both dynamics and geometry \cite{Helene1}. We hope that in the future we shall be able to further comment both on the extension of the results stated above in the multidimensional case and on the corresponding results in the case of diffeomorphisms.

\end{document}